\documentclass[a4paper,oneside]{article}
\usepackage[utf8]{inputenc}
\usepackage[a4paper,margin=3cm]{geometry} 
\usepackage{amsmath}
\usepackage{amsthm}
\usepackage{amscd}
\usepackage{amssymb}
\usepackage{MnSymbol}
\usepackage{latexsym}
\usepackage{eucal}
\usepackage{dsfont}
\usepackage{mathtools}
\usepackage{enumitem}
\usepackage{verbatim}
\usepackage{graphicx}
\usepackage{float}

\usepackage{pdflscape}

\usepackage[colorlinks,pdftex]{hyperref}
\hypersetup{
linkcolor=black,
citecolor=black,
pdftitle={},
pdfauthor={Franziska K\"uhn},
pdfkeywords={},
}

\makeatletter
\renewcommand\section{\@startsection{section}{1}{0mm}{-1.5\baselineskip}{\baselineskip}{\normalsize\bfseries\sffamily}}
\renewcommand\subsection{\@startsection{subsection}{1}{0mm}{-\baselineskip}{\baselineskip}{\normalsize\bfseries\sffamily}}
\makeatother

\makeatletter
\def\@fnsymbol#1{\ensuremath{\ifcase#1\or *\or **\or \dagger\or \ddagger\or
   \mathsection\or \mathparagraph\or \|\or \dagger\dagger
   \or \ddagger\ddagger \else\@ctrerr\fi}}

\newlength{\preskip}
\newlength{\postskip}
\makeatother

\newtheoremstyle{theorem}{\preskip}{\postskip}{\itshape}{}{\bfseries}{}
{.5em}{\textbf{\thmname{#1}\thmnumber{ #2} (\thmnote{ #3})}}
\newtheoremstyle{definition}{\preskip}{\postskip}{\normalfont}{0pt}{\bfseries}{}{.5em}{}
\newtheoremstyle{remark}{\preskip}{\postskip}{\normalfont}{0pt}{\bfseries}{}{.5em}{}

\swapnumbers
\theoremstyle{theorem} \newtheorem{thm}{Theorem}[section]
\theoremstyle{theorem} \newtheorem{lem}[thm]{Lemma}
\theoremstyle{theorem} 
\theoremstyle{theorem} \newtheorem{kor}[thm]{Corollary}
\theoremstyle{definition} 
\theoremstyle{remark} \newtheorem{bem_thm}[thm]{Remark}
\theoremstyle{remark} 
\theoremstyle{definition} 
\theoremstyle{definition} \newtheorem*{ack}{Acknowledgements}
\theoremstyle{remark} 
\theoremstyle{remark} 
\theoremstyle{definition}  \newtheorem{bsp}[thm]{Example}
\theoremstyle{definition}  
\theoremstyle{definition} 
\DeclareMathOperator \re {Re}
\DeclareMathOperator \im {Im}

\DeclareMathOperator \Exp {Exp}

\DeclareMathOperator \loc {loc}

\newcommand{\I}{\mathds{1}}

\newcommand\fa{\qquad \text{for all \ }}

\newcommand{\cadlag}{c\`adl\`ag }

\newcommand\mc[1] {\mathcal{#1}}
\newcommand\mbb[1] {\mathds{#1}}

\newcommand{\eps}{\varepsilon}

\author{%
    Franziska K\"{u}hn\thanks{Institut f\"ur Mathematische Stochastik, Fachrichtung Mathematik, Technische Universit\"at Dresden, 01062 Dresden, Germany, \texttt{franziska.kuehn1@tu-dresden.de}} 
}

\title{On Martingale Problems and Feller Processes}

\date{}

\begin{document}

\maketitle

\abstract{\noindent Let $A$ be a pseudo-differential operator with negative definite symbol $q$. In this paper we establish a sufficient condition such that the well-posedness of the $(A,C_c^{\infty}(\mbb{R}^d))$-martingale problem implies that the unique solution to the martingale problem is a Feller process. This provides a proof of a former claim by van Casteren. As an application we prove new existence and uniqueness results for L\'evy-driven stochastic differential equations and stable-like processes with unbounded coefficients. \par \medskip

 \par \medskip

}

\section{Introduction} \label{intro}

Let $(L_t)_{t \geq 0}$ be a $k$-dimensional L\'evy process with characteristic exponent $\psi: \mbb{R}^d \to \mbb{C}$ and $\sigma: \mbb{R}^d \to \mbb{R}^{d \times k}$ a continuous function which is at most of linear growth. It is known that there is a intimate correspondence between the L\'evy-driven stochastic differential equation (SDE) \begin{equation}
	dX_t = \sigma(X_{t-}) \, dL_t, \qquad X_0 \sim \mu, \label{sde0}
\end{equation}
and the pseudo-differential operator $A$ with symbol $q(x,\xi) := \psi(\sigma(x)^T \xi)$, i.\,e.\ \begin{equation*}
	Af(x) = - \int_{\mbb{R}^d} q(x,\xi) e^{ix \cdot \xi} \hat{f}(\xi) \, d\xi, \qquad f \in C_c^{\infty}(\mbb{R}^d), \, x \in \mbb{R}^d,
\end{equation*}
where $\hat{f}$ denotes the Fourier transform of a smooth function $f$ with compact support. Kurtz \cite{kurtz} proved that the existence of a unique weak solution to the SDE for any initial distribution $\mu$ is equivalent to the well-posedness of the $(A,C_c^{\infty}(\mbb{R}^d))$-martingale problem. Recently, we have shown in \cite{sde} that a unique solution to the martingale problem -- or, equivalently, to the SDE \eqref{sde0} -- is a Feller process if the L\'evy measure $\nu$ satisfies \begin{equation*}
	\nu(\{y \in \mbb{R}^k; |\sigma(x) \cdot y+x| \leq r\}) \xrightarrow[]{|x| \to \infty} 0 \fa r>0
\end{equation*}
which is equivalent to saying that $A$ maps $C_c^{\infty}(\mbb{R}^d)$ into $C_{\infty}(\mbb{R}^d)$, the space of continuous functions vanishing at infinity. \par
In this paper, we are interested in the following more general question: Consider a pseudo-differential operator $A$ with continuous negative definite symbol $q$,\begin{equation*}
	q(x,\xi) = q(x,0) -ib(x) \cdot \xi + \frac{1}{2} \xi \cdot Q(x) \xi + \int_{y \neq 0} (1-e^{iy \cdot \xi}+iy \cdot \xi \I_{(0,1)}(|y|)) \, \nu(x,dy), \quad x,\xi \in \mbb{R}^d,
\end{equation*}
such that the $(A,C_c^{\infty}(\mbb{R}^d))$-martingale problem is well-posed, i.\,e.\ for any initial distribution $\mu$ there exists a unique solution to the $(A,C_c^{\infty}(\mbb{R}^d))$-martingale problem. Under which assumptions does the well-posedness of the $(A,C_c^{\infty}(\mbb{R}^d))$-martingale problem imply that the unique solution to the martingale problem is a Feller process? Since the infinitesimal generator of the solution is, when restricted to $C_c^{\infty}(\mbb{R}^d)$, the pseudo-differential operator $A$, it is clear that $A$ has to satisfy $Af \in C_{\infty}(\mbb{R}^d)$ for all $f \in C_c^{\infty}(\mbb{R}^d)$. In a paper by van Casteren \cite{cast1} it was claimed that this mapping property of $A$ already implies that the solution is a Feller process; however, this result turned out to be wrong, see \cite[Example 2.27(ii)]{ltp} for a counterexample. Our main result states van Casteren's claim is \emph{correct} if the symbol $q$ satisfies a certain growth condition; the required definitions will be explained in Section~\ref{def}.

\begin{thm} \label{1.1}
	Let $A$ be a pseudo-differential operator with continuous negative definite symbol $q$ such that $q(\cdot,0)=0$ and $A$ maps $C_c^{\infty}(\mbb{R}^d)$ into $C_{\infty}(\mbb{R}^d)$.  If the $(A,C_c^{\infty}(\mbb{R}^d))$-martingale problem is well-posed and \begin{equation}
			\lim_{|x| \to \infty} \sup_{|\xi| \leq |x|^{-1}} |q(x,\xi)| < \infty, \label{lin-grow} \tag{G}
		\end{equation}
		then the solution $(X_t)_{t \geq 0}$ to the martingale problem is a conservative rich Feller process with symbol $q$.
\end{thm}

\begin{bem_thm} \label{1.3} \begin{enumerate}
	\item If the martingale problem is well-posed and $A(C_c^{\infty}(\mbb{R}^d)) \subseteq C_{\infty}(\mbb{R}^d)$, then the solution is a $C_b$-Feller process, i.\,e.\ the associated semigroup $(T_t)_{t \geq 0}$ satisfies $T_t: C_b(\mbb{R}^d) \to C_b(\mbb{R}^d)$ for all $t \geq 0$. The growth condition \eqref{lin-grow} is needed to prove the Feller property; that is, to show that $T_t f$ vanishes at infinity for any $f \in C_{\infty}(\mbb{R}^d)$ and $t \geq 0$.
	\item There is a partial converse to Theorem~\ref{1.1}: If $(X_t)_{t \geq 0}$ is a Feller process and $C_c^{\infty}(\mbb{R}^d)$ is a core for the generator $A$ of $(X_t)_{t \geq 0}$, then the $(A,C_c^{\infty}(\mbb{R}^d))$-martingale problem is well-posed, see e.\,g.\ \cite[Theorem 4.10.3]{kol} or \cite[Theorem 1.37]{matters} for a proof.
	\item The mapping property $A(C_c^{\infty}(\mbb{R}^d)) \subseteq C_{\infty}(\mbb{R}^d)$ can be equivalently formulated in terms of the symbol $q$ and its characteristics, cf.\ Lemma~\ref{map}.
	\item For the particular case that $A$ is the pseudo-differential operator associated with the SDE \eqref{sde0}, i.\,e.\ $q(x,\xi) = \psi(\sigma(x)^T \xi)$, we recover \cite[Theorem 1.1]{sde}. Note that the growth condition \eqref{lin-grow} is automatically satisfied for any function $\sigma$ which is at most of linear growth.
\end{enumerate} \end{bem_thm}

Although it is, in general, hard to prove the well-posedness of a martingale problem, Theorem~\ref{1.1} is very useful since it allows us to use localization techniques for martingale problems to establish new existence results for Feller processes with unbounded coefficients.

\begin{kor} \label{1.2}
	Let $A$ be a pseudo-differential operator with symbol $q$ such that $q(\cdot,0)=0$, $A(C_c^{\infty}(\mbb{R}^d)) \subseteq C_{\infty}(\mbb{R}^d)$ and \begin{equation*}
		\lim_{|x| \to \infty} \sup_{|\xi| \leq |x|^{-1}} |q(x,\xi)|<\infty.
	\end{equation*}
	Assume that there exists a sequence $(q_k)_{k \in \mbb{N}}$ of symbols such that $q_k(x,\xi) = q(x,\xi)$ for all $|x| <k$, $\xi \in \mbb{R}^d$, and the pseudo-differential operator $A_k$ with symbol $q_k$ maps $C_c^{\infty}(\mbb{R}^d)$ into $C_{\infty}(\mbb{R}^d)$. If the $(A_k,C_c^{\infty}(\mbb{R}^d))$-martingale problem is well posed for all $k \geq 1$, then there exists conservative rich Feller process $(X_t)_{t \geq 0}$ with symbol $q$, and $(X_t)_{t \geq 0}$ is the unique solution to the $(A,C_c^{\infty}(\mbb{R}^d))$-martingale problem.
\end{kor}

The paper is organized as follows. After introducing basic notation and definitions in Section~\ref{def}, we prove Theorem~\ref{1.1} and Corollary~\ref{1.2}. In Section~\ref{app} we present applications and examples; in particular we obtain new existence and uniqueness results for L\'evy-driven stochastic differential equations and stable-like processes with unbounded coefficients.


\section{Preliminaries} \label{def}

We consider $\mbb{R}^d$ endowed with the Borel $\sigma$-algebra $\mc{B}(\mbb{R}^d)$ and write $B(x,r)$ for the open ball centered at $x \in \mbb{R}^d$ with radius $r>0$; $\mbb{R}^d_{\Delta}$ is the one-point compactification of $\mbb{R}^d$. If a certain statement holds for $x \in \mbb{R}^d$ with $|x|$ sufficiently large, we write ``for $|x| \gg 1$''. For a metric space $(E,d)$ we denote by $C(E)$ the space of continuous functions $f: E \to \mbb{R}$; $C_{\infty}(E)$ (resp.\ $C_b(E)$) is the space of continuous functions which vanish at infinity (resp.\ are bounded). A function $f: [0,\infty) \to E$ is in the Skorohod space $D([0,\infty),E)$ if $f$ is right-continuous and has left-hand limits in $E$. We will always consider $E=\mbb{R}^d$ or $E=\mbb{R}^d_{\Delta}$.  \par
An $E$-valued Markov process $(\Omega,\mc{A},\mbb{P}^x,x \in E,X_t,t \geq 0)$ with \cadlag (right-continuous with left-hand limits) sample paths is called a \emph{Feller process} if the associated semigroup $(T_t)_{t \geq 0}$ defined by \begin{equation*}
	T_t f(x) := \mbb{E}^x f(X_t), \quad x \in E, f \in \mc{B}_b(E) := \{f: E \to \mbb{R}; \text{$f$ bounded, Borel measurable}\}
\end{equation*}
has the \emph{Feller property}, i.\,e.\ $T_t f \in C_{\infty}(E)$ for all $f \in C_{\infty}(E)$, and $(T_t)_{t \geq 0}$ is \emph{strongly continuous at $t=0$}, i.\,e. $\|T_tf-f\|_{\infty} \xrightarrow[]{t \to 0} 0$ for any $f \in C_{\infty} (E)$. Following \cite{rs98} we call a Markov process $(X_t)_{t \geq 0}$ with \cadlag sample paths a \emph{$C_b$-Feller process} if $T_t(C_b(E)) \subseteq C_b(E)$ for all $t \geq 0$. An $\mbb{R}^d_{\Delta}$-valued Markov process with semigroup $(T_t)_{t \geq 0}$ is \emph{conservative} if $T_t \I_{\mbb{R}^d} = \I_{\mbb{R}^d}$ for all $t \geq 0$. \par
If the smooth functions with compact support $C_c^{\infty}(\mbb{R}^d)$ are contained in the domain of the generator $(L,\mc{D}(L))$ of a Feller process $(X_t)_{t \geq 0}$, then we speak of a \emph{rich} Feller process. A result due to von Waldenfels and Courr\`ege, cf.\ \cite[Theorem 2.21]{ltp}, states that the generator $L$ of an $\mbb{R}^d$-valued rich Feller process is, when restricted to $C_c^{\infty}(\mbb{R}^d)$, a pseudo-differential operator with negative definite symbol: \begin{equation*}
	Lf(x) =  - \int_{\mbb{R}^d} e^{i \, x \cdot \xi} q(x,\xi) \hat{f}(\xi) \, d\xi, \qquad f \in C_c^{\infty}(\mbb{R}^d), \, x \in \mbb{R}^d
\end{equation*}
where $\hat{f}(\xi) := \mc{F}f(\xi):= (2\pi)^{-d} \int_{\mbb{R}^d} e^{-ix \cdot \xi} f(x) \, dx$ denotes the Fourier transform of $f$ and \begin{equation}
	q(x,\xi) = q(x,0) - i b(x) \cdot \xi + \frac{1}{2} \xi \cdot Q(x) \xi + \int_{\mbb{R}^d \backslash \{0\}} (1-e^{i y \cdot \xi}+ i y\cdot \xi \I_{(0,1)}(|y|)) \, \nu(x,dy). \label{cndf}
\end{equation}
We call $q$ the \emph{symbol} of the Feller process $(X_t)_{t \geq 0}$ and of the pseudo-differential operator; $(b,Q,\nu)$ are the \emph{characteristics} of the symbol $q$. For each fixed $x \in \mbb{R}^d$, $(b(x),Q(x),\nu(x,dy))$ is a L\'evy triplet, i.\,e.\ $b(x) \in \mbb{R}^d$, $Q(x) \in \mbb{R}^{d \times d}$ is a symmetric positive semidefinite matrix and $\nu(x,dy)$ a $\sigma$-finite measure on $(\mbb{R}^d \backslash \{0\},\mc{B}(\mbb{R}^d \backslash \{0\}))$ satisfying $\int_{y \neq 0} \min\{|y|^2,1\} \, \nu(x,dy)<\infty$. We use $q(x,D)$ to denote the pseudo-differential operator $L$ with continuous negative definite symbol $q$. A family of continuous negative definite functions $(q(x,\cdot))_{x \in \mbb{R}^d}$ is \emph{locally bounded} if for any compact set $K \subseteq \mbb{R}^d$ there exists $c>0$ such that $|q(x,\xi)| \leq c(1+|\xi|^2)$ for all $x \in K$, $\xi \in \mbb{R}^d$.  By \cite[Lemma 2.1, Remark 2.2]{rs-grow}, this is equivalent to \begin{equation}
		\forall K \subseteq \mbb{R}^d \, \, \text{cpt.}: \, \, \, \sup_{x \in K} |q(x,0)| +\sup_{x \in K} |b(x)| + \sup_{x \in K} |Q(x)|+ \sup_{x \in K} \int_{y \neq 0} (|y|^2 \wedge 1) \, \nu(x,dy) <\infty.  \label{loc-bdd}
\end{equation}
If \eqref{loc-bdd} holds for $K=\mbb{R}^d$, we say that $q$ has bounded coefficients. We will frequently use the following result.

\begin{lem} \label{map}
	Let $L$ be a pseudo-differential operator with continuous negative definite symbol $q$ and characteristics $(b,Q,\nu)$. Assume that $q(\cdot,0)=0$ and that $q$ is locally bounded. \begin{enumerate}
		\item\label{map-i} $\lim_{|x| \to \infty} Lf(x) = 0$ for all $f \in C_c^{\infty}(\mbb{R}^d)$ if, and only if, \begin{equation}
			\lim_{|x| \to \infty} \nu(x,B(-x,r)) = 0 \fa r>0. \label{def-eq5}
		\end{equation}
		\item\label{map-ii} If $\lim_{|x| \to \infty} \sup_{|\xi| \leq |x|^{-1}} |\re q(x,\xi)| = 0$, then \eqref{def-eq5} holds.
		\item\label{map-iii} $L(C_c^{\infty}(\mbb{R}^d)) \subseteq C(\mbb{R}^d)$ if, and only if, $x \mapsto q(x,\xi)$ is continuous for all $\xi \in \mbb{R}^d$.
	\end{enumerate}
\end{lem}

For a proof of Lemma~\ref{map}\ref{map-i},\ref{map-ii} see \cite[Lemma 3.26]{ltp} or \cite[Theorem 1.27]{diss}; \ref{map}\ref{map-iii} goes back to Schilling \cite[Theorem 4.4]{rs98}, see also \cite[Theorem A.1]{change}. \par
If the symbol $q$ of a rich Feller process $(L_t)_{t \geq 0}$ does not depend on $x$, i.\,e.\ $q(x,\xi) = q(\xi)$, then $(L_t)_{t \geq 0}$ is a \emph{L\'evy process}. This is equivalent to saying that $(L_t)_{t \geq 0}$ has stationary and independent increments and \cadlag sample paths. The symbol $q=q(\xi)$ is called \emph{characteristic exponent}. Our standard reference for L\'evy processes is the monograph \cite{sato} by Sato. \emph{Weak uniqueness} holds for the \emph{L\'evy-driven stochastic differential equation} (SDE, for short) \begin{equation*}
	dX_t = \sigma(X_{t-}) \, dL_t, \qquad X_0 \sim \mu,
\end{equation*}
if any two weak solutions of the SDE have the same finite-dimensional distributions. We refer the reader to Situ \cite{situ} for further details. \par
Let $(A,\mc{D})$ be a linear operator with domain $\mc{D} \subseteq \mc{B}_b(E)$ and $\mu$ a probability measure on $(E,\mc{B}(E))$. A $d$-dimensional stochastic process $(X_t)_{t \geq 0}$, defined on a probability space $(\Omega,\mc{A},\mbb{P}^{\mu})$, with \cadlag sample paths is a \emph{solution to the $(A,\mc{D})$-martingale problem with initial distribution $\mu$}, if $X_0 \sim \mu$ and \begin{equation*}
	M_t^f := f(X_t)-f(X_0)- \int_0^t Af(X_s) \, ds, \qquad t \geq 0,
\end{equation*}
is a $\mbb{P}^{\mu}$-martingale with respect to the canonical filtration of $(X_t)_{t \geq 0}$ for any $f \in \mc{D}$. By considering the measure $\mbb{Q}^{\mu}$ induced by $(X_t)_{t \geq 0}$ on $D([0,\infty),E)$ we may assume without loss of generality that $\Omega = D([0,\infty),E)$ is the Skorohod space and $X_t(\omega) := \omega(t)$ the canonical process. The $(A,\mc{D})$-martingale problem is \emph{well-posed} if for any initial distribution $\mu$ there exists a unique (in the sense of finite-dimensional distributions) solution to the $(A,\mc{D})$-martingale problem with initial distribution $\mu$. For a comprehensive study of martingale problems see \cite[Chapter 4]{ethier}.

\section{Proof of the main results} \label{p}

In order to prove Theorem~\ref{1.1} we need the following statement which allows us to formulate the linear growth condition \eqref{lin-grow} in terms of the characteristics.

\begin{lem} \label{p-3}
	Let $(q(x,\cdot))_{x \in \mbb{R}^d}$ be a family of continuous negative definite functions with characteristics $(b,Q,\nu)$ such that $q(\cdot,0)=0$. Then \begin{equation}
		\limsup_{|x| \to \infty} \sup_{|\xi| \leq |x|^{-1}} |q(x,\xi)| < \infty \tag{G}
	\end{equation}
	if, and only if,  there exists an absolute constant $c>0$ such that each of the following conditions is satisfied for $|x| \gg 1$.\begin{enumerate}
		\item\label{p-3-i} $\left| b(x) + \int_{1 \leq |y|<|x|/2} y \, \nu(x,dy) \right| \leq c(1+|x|)$.
		\item\label{p-3-ii} $|Q(x)| + \int_{|y| \leq |x|/2} |y|^2 \, \nu(x,dy) \leq c(1+|x|^2)$.
		\item\label{p-3-iii} $\nu(x, \{y \in \mbb{R}^d; |y| \geq  1 \vee |x|/2\}) \leq c$.
	\end{enumerate}
	If \eqref{lin-grow} holds and $q$ is locally bounded, cf.\ \eqref{loc-bdd}, then \ref{p-3-i}--\ref{p-3-iii} hold for all $x \in \mbb{R}^d$.
\end{lem}

\begin{proof}
	First we prove that \ref{p-3-i}--\ref{p-3-iii} are sufficient for \eqref{lin-grow}. Because of \ref{p-3-i} and \ref{p-3-ii} it suffices to show that \begin{equation*}
		p(x,\xi) := \int_{y \neq 0} (1-e^{iy \cdot \xi}+iy \cdot \xi \I_{(0,|x|/2)}(|y|)) \, \nu(x,dy)
	\end{equation*}
	satisfies the linear growth condition \eqref{lin-grow}. Using the elementary estimates \begin{equation*}
		|1-e^{iy \cdot \xi}| \leq 2 \qquad \text{and} \qquad |1-e^{iy \cdot \xi}+iy \cdot \xi| \leq \frac{1}{2} |\xi|^2 |y|^2
	\end{equation*}
	we find \begin{align*}
		|p(x,\xi)|
		\leq \frac{|\xi|^2}{2} \int_{0<|y| < |x|/2} |y|^2 \, \nu(x,dy) + 2 \int_{|y| \geq |x|/2} \, \nu(x,dy)
	\end{align*}
	for all $|x| \geq 1$ which implies by \ref{p-3-ii} and \ref{p-3-iii} that \begin{equation*}
		\limsup_{|x| \to \infty} \sup_{|\xi| \leq |x|^{-1}} |p(x,\xi)| < \infty.
	\end{equation*}
	It remains to prove that \eqref{lin-grow} implies \ref{p-3-i}-\ref{p-3-iii}. For \ref{p-3-ii} and \ref{p-3-iii} we use a similar idea as in \cite[proof of Theorem 4.4]{rs98}. It is known that the function $g$ defined by \begin{equation*}
		g(\eta) := \frac{1}{2} \int_{(0,\infty)} \frac{1}{(2\pi r)^{d/2}} \exp \left( - \frac{|\eta|^2}{2r} - \frac{r}{2} \right) \, dr, \qquad \eta \in \mbb{R}^d,
	\end{equation*}
	has a finite second moment, i.\,e.\ $\int_{\mbb{R}^d} |\eta|^2 g(\eta) \, d\eta<\infty$, and satisfies \begin{equation}
		\frac{|z|^2}{1+|z|^2} = \int_{\mbb{R}^d} (1-\cos(\eta \cdot z)) g(\eta) \, d\eta \label{p-eq7}
	\end{equation}
	for all $z \in \mbb{R}^d$. As \begin{equation*}
		\inf_{|z| \geq 1/2} \frac{|z|^2}{1+|z|^2} = \frac{1}{5}>0
	\end{equation*}
	we obtain by applying Tonelli's theorem \begin{align*}
		\frac{1}{5} \nu(x; \{y; |y| \geq |x|/2\})
		\leq \int_{|y| \geq |x|/2} \frac{\left( \frac{|y|}{|x|} \right)^2}{1+ \left( \frac{|y|}{|x|} \right)^2} \, \nu(x,dy)
		&= \int_{|y| \geq |x|/2} \int_{\mbb{R}^d} \left(1- \cos \frac{\eta \cdot y}{|x|} \right) g(\eta) \, d\eta \, \nu(x,dy) \\
		&\leq \int_{\mbb{R}^d} \re q \left( x,\frac{\eta}{|x|} \right) \, d\eta.
	\end{align*}
	Since \begin{equation*}
		|\psi(\xi)| \leq 2 \sup_{|\zeta| \leq 1} |\psi(\zeta)| (1+|\xi|^2), \qquad \xi \in \mbb{R}^d,
	\end{equation*}
	for any continuous negative definite function $\psi$, cf.\ \cite[Proposition 2.17d)]{ltp}, we get \begin{align*}
		\nu(x; \{y; |y| \geq |x|/2\})
		&\leq 10 \sup_{|\xi| \leq 1} \left| q \left(x, \frac{\xi}{|x|} \right) \right| \int_{\mbb{R}^d} (1+|\eta|^2) g(\eta) \, d\eta,
	\end{align*}
	and this gives \ref{p-3-iii} for $|x| \gg 1$. Next we prove \ref{p-3-ii}. First of all, we note that \begin{equation*}
		0 \leq \xi \cdot Q(x) \xi \leq \re q(x,\xi) \leq |q(x,\xi)|
	\end{equation*}
	and therefore $|Q(x)| \leq c (1+|x|^2)$ is a direct consequence of \eqref{lin-grow}. On the other hand, \begin{equation*}
		\inf_{|y| \leq |x|/2} \frac{1}{|x|^2+|y|^2} \geq \frac{4}{5} \frac{1}{|x|^2}
	\end{equation*}
	implies that \begin{align*}
		\frac{4}{5} \frac{1}{|x|^2} \int_{|y| \leq |x|/2} |y|^2 \, \nu(x,dy)
		\leq \int_{|y| \leq |x|/2} \frac{|y|^2}{|x|^2+|y|^2} \, \nu(x,dy)
		= \int_{|y| \leq |x|/2} \frac{\left( \frac{|y|}{|x|} \right)^2}{1+ \left( \frac{|y|}{|x|} \right)^2} \, \nu(x,dy).
	\end{align*}
	Using \eqref{p-eq7} and applying Tonelli's theorem once more, we find \begin{align*}
		\int_{|y| \leq |x|/2} |y|^2 \, \nu(x,dy)
		&\leq \frac{5}{4} |x|^2 \int_{\mbb{R}^d} \re q \left( x, \frac{\eta}{|x|} \right) g(\eta) \, d\eta.
	\end{align*}
	Hence, \begin{align*}
		\int_{|y| \leq |x|/2} |y|^2 \, \nu(x,dy)
		\leq \frac{5}{4} |x|^2 \sup_{|\xi| \leq 1} \left| q \left(x, \frac{\xi}{|x|} \right) \right| \int_{\mbb{R}^d} (1+|\eta|^2) g(\eta) \, d\eta
	\end{align*}
	and \ref{p-3-ii} follows. Finally, as \ref{p-3-ii} and \ref{p-3-iii} imply that \begin{equation*}
		\limsup_{|x| \to \infty} \sup_{|\xi| \leq |x|^{-1}} \left| q(x,\xi) - i \xi \cdot \left( b(x) + \int_{1 \leq |y| < |x|/2} y \, \nu(x,dy) \right) \right| < \infty,
	\end{equation*}
	see the first part of the proof, a straightforward application of the triangle inequality gives \begin{align*}
		\limsup_{|x| \to \infty} \sup_{|\xi| \leq |x|^{-1}} \left| i \xi \cdot \left( b(x) + \int_{1 \leq |y|< |x|/2} y \, \nu(x,dy) \right) \right| < \infty
	\end{align*}
	which proves \ref{p-3-i}.
\end{proof}

\begin{kor} \label{p-5}
	Let $A$ be a pseudo-differential operator with continuous negative definite symbol $q$ such that $q(\cdot,0)=0$. If $A$ maps $C_c^{\infty}(\mbb{R}^d)$ into $C_{\infty}(\mbb{R}^d)$ and $q$ satisfies the linear growth condition \eqref{lin-grow}, then there exists for any initial distribution $\mu$ a solution to the $(A,C_c^{\infty}(\mbb{R}^d))$-martingale problem which is conservative.
\end{kor}

\begin{proof}
	Since $A(C_c^{\infty}(\mbb{R}^d)) \subseteq C_{\infty}(\mbb{R}^d)$ and $A$ satisfies the positive maximum principle, it follows from \cite[ Theorem 4.5.4]{ethier} that there exists an $\mbb{R}^d_{\Delta}$-valued solution to the $(A,C_c^{\infty}(\mbb{R}^d))$-martingale problem with initial distribution $\mu:=\delta_x$. By considering the probability measure induced by $(X_t)_{t \geq 0}$ on the Skorohod space $D([0,\infty),\mbb{R}^d_{\Delta})$, we may assume without loss of generality that $X_t(\omega) := \omega(t)$ is the canonical process on $\Omega := D([0,\infty),\mbb{R}^d_{\Delta})$. 	Lemma~\ref{p-3} implies that \begin{equation*}
		\lim_{r \to \infty} \sup_{|z-x| \leq 2r} \sup_{|\xi| \leq r^{-1}} |q(z,\xi)| < \infty \fa x \in \mbb{R}^d,
	\end{equation*}
	and therefore \cite[Corollary 3.2]{change} shows that the solution with initial distribution $\delta_x$ does not explode in finite time with probability $1$. By construction, see \cite[proof of Theorem 4.5.4]{ethier}, the mapping $x \mapsto \mbb{P}^x(B)$ is measurable for all $B \in \mc{F}_{\infty}^X := \sigma(X_t; t \geq 0)$. If we define \begin{equation*}
		\mbb{P}^{\mu}(B) := \int_{\mbb{R}^d} \mbb{P}^x(B) \, \mu(dx), \qquad B \in \mc{F}_{\infty}^X 
	\end{equation*}
	then $\mbb{P}^{\mu}$ gives rise to a conservative solution to the $(A,C_c^{\infty}(\mbb{R}^d))$-martingale problem with initial condition $\mu$.
\end{proof}
	
In Section~\ref{app} we will formulate Corollary~\ref{p-5} for solutions of stochastic differential equations, cf.\ Theorem~\ref{app-0}. The next result is an important step to prove Theorem~\ref{1.1}. 

\begin{lem} \label{p-7}
	Let $L$ be a pseudo-differential operator with continuous negative definite symbol $p$ and characteristics $(b,Q,\nu)$ such that $p(\cdot,0)=0$ and $L(C_c^{\infty}(\mbb{R}^d)) \subseteq C_{\infty}(\mbb{R}^d)$. Assume that $\nu(x, \{y \in \mbb{R}^d; |y| \geq |x|/2\}) = 0$ for $|x| \gg 1 $ and \begin{equation}
		\limsup_{|x| \to \infty} \sup_{|\xi| \leq |x|^{-1}} |p(x,\xi)| < \infty. \tag{G}
	\end{equation}
\begin{enumerate}
	\item\label{p-7-i} For any initial distribution $\mu$ there exists a probability measure $\mbb{P}^{\mu}$ on $D([0,\infty),\mbb{R}^d)$ such that the canonical process $(Y_t)_{t \geq 0}$ solves the $(L,C_c^{\infty}(\mbb{R}^d))$-martingale problem and \begin{equation}
		\mbb{P}^{\mu}(B) = \int \mbb{P}^x(B) \, \mu(dx) \fa B \in \mc{F}_{\infty}^Y := \sigma(Y_t; t \geq 0).  \label{p-eq8}
	\end{equation}
	\item\label{p-7-ii} For any $t \geq 0$, $R>0$ and $\eps>0$ there exist constants $\varrho>0$ and $\delta>0$ such that \begin{equation}
		\mbb{P}^{\mu} \left( \inf_{s \leq t} |Y_s| < R \right) \leq \eps \label{p-eq9}
	\end{equation}
	for any initial distribution $\mu$ such that $\mu(B(0,\varrho)) \leq \delta$.
	\item\label{p-7-iii} For any $t \geq 0$, $\eps>0$ and any compact set $K \subseteq \mbb{R}^d$ there exists $R>0$ such that \begin{equation}
		\mu(K^c) \leq \frac{\eps}{2} \implies \mbb{P}^{\mu} \left( \sup_{s \leq t} |Y_s| \geq R \right) \leq \eps. \label{p-eq10}
	\end{equation}
\end{enumerate} \end{lem}

\begin{proof} 
	\ref{p-7-i} is a direct consequence of Corollary~\ref{p-5}; we have to prove \ref{p-7-ii} and \ref{p-7-iii}. To keep notation simple we show the result only in dimension $d=1$. Since $L$ maps $C_c^{\infty}(\mbb{R})$ into $C_{\infty}(\mbb{R})$, the symbol $p$ is locally bounded, cf.\ \cite[Proposition 2.27(d)]{ltp}, and therefore Lemma~\ref{p-3} shows that \ref{p-3}\ref{p-3-i}--\ref{p-3-iii} hold for all $x \in \mbb{R}$. Set $u(x) := 1/(1+|x|^2)$, $x \in \mbb{R}$, then \begin{equation}
		|u'(x)| \leq 2|x| u(x)^2 \qquad \text{and} \qquad |u''(x)| \leq 6u(x)^2 \fa x \in \mbb{R}. \label{p-eq11}
	\end{equation}
	Clearly, $|Lu(x)| \leq I_1+I_2$ where \begin{align*}	
		I_1 &:= \left| b(x) + \int_{1 \leq |y| < |x|/2} y \, \nu(x,dy) \right| \, |u'(x)| + \frac{1}{2} |Q(x)| \, |u''(x)| \\
		I_2 &:= \left| \int_{y<|x/2} (u(x+y)-u(x)-u'(x) y) \, \nu(x,dy) \right|
	\end{align*}
	for all $|x| \gg 1$. By Lemma~\ref{p-3} and \eqref{p-eq11} there exists a constant $c_1>0$ such that $I_1 \leq c_1 u(x)$ for all $x \in \mbb{R}$. On the other hand, Taylor's formula shows \begin{equation*}
		I_2 \leq \frac{1}{2} \int_{|y|<|x|/2} |y|^2 \, |u''(\zeta)| \, \nu(x,dy)
	\end{equation*}
	for some intermediate value $\zeta=\zeta(x,y)$ between $x$ and $x+y$. Since $|y| < |x|/2$, we have $|\zeta| \geq |x|/2$; hence, by \eqref{p-eq11}, \begin{equation*}
		|u''(\zeta)| \leq 6 u(\zeta)^2 \leq 24 u(x)^2.
	\end{equation*}
	Applying Lemma~\ref{p-3}, we find that there exists a constant $c_2>0$ such that \begin{equation*}
		I_2 \leq  24 u(x)^2 \int_{|y| < |x|/2} |y|^2 \, \nu(x,dy) \leq c_2 u(x).
	\end{equation*}
	Consequently, $|Lu(x)| \leq (c_1+c_2) u(x)$ for all $|x| \gg 1$. As $Lu$ is bounded and $u$ is bounded away from $0$ on compact sets, we can choose a constant $c_3>0$ such that \begin{equation}
		|Lu(x)| \leq c_3 u(x) \fa x \in \mbb{R}. \tag{$\star$} \label{p-star1}
	\end{equation}
	Define $\tau_R := \inf\{t \geq 0; |Y_t|<R\}$. Using a standard truncation and stopping technique it follows that \begin{equation*}
		\mbb{E}^{\mu} u(Y_{t \wedge \tau_R})-\mbb{E}^{\mu} u(Y_0) = \mbb{E}^{\mu} \left( \int_{(0,t \wedge \tau_R)} Lu(Y_s) \, ds \right).
	\end{equation*}
	Hence, by \eqref{p-star1}, \begin{equation*}
		\mbb{E}^{\mu} u(Y_{t \wedge \tau_R}) \leq \mbb{E}^{\mu} u(Y_0) + c_3 \mbb{E}^{\mu} \left( \int_{(0,t)} u(Y_{s \wedge \tau_R}) \, ds \right).
	\end{equation*}
	An application of Gronwall's inequality shows that there exists a constant $C>0$ such that \begin{equation*}
		\mbb{E}^{\mu} u(Y_{t \wedge \tau_R}) \leq e^{Ct} \mbb{E}^{\mu}u(Y_0) \fa t \geq 0.
	\end{equation*}
	By the Markov inequality, this implies that \begin{align*}
		\mbb{P}^{\mu} \left( \inf_{s \leq t} |Y_s|<R \right)
		\leq \mbb{P}^{\mu}(|Y_{t \wedge \tau_R}| \leq R)
		\leq \mbb{P}^{\mu}\big(u(Y_{t \wedge \tau_R}) \geq u(R) \big) 
		&\leq \frac{1}{u(R)} \mbb{E}^{\mu}u(Y_{t \wedge \tau_R}) \\
		&\leq \frac{1}{u(R)} e^{Ct} \mbb{E}^{\mu} u(Y_0).
	\end{align*}
	If $\mu$ is an initial distribution such that $\mu(B(0,\varrho)) \leq \delta$, then $\mbb{E}^{\mu} u(Y_0) \leq \delta + \varrho^{-2}$. Choosing $\varrho$ sufficiently large and $\delta>0$ sufficiently small, we get \eqref{p-eq9}.	The proof of \ref{p-3-iii} is similar. If we set $v(x) := x^2+1$, then there exists by Lemma~\ref{p-3} a constant $c>0$ such that $|Lv(x)| \leq c v(x)$ for all $x \in \mbb{R}$. Applying Gronwall's inequality another time, we find a constant $C>0$ such that \begin{equation*}
		\mbb{E}^{\mu} v(Y_{t \wedge \sigma_R}) \leq e^{Ct} \mbb{E}^{\mu}v(Y_0), \qquad t \geq 0,
	\end{equation*}
	where $\sigma_R := \inf\{t \geq 0; |Y_t| \geq R\}$ denotes the exit time from the ball $B(0,R)$. Hence, by the Markov inequality, \begin{equation*}
		\mbb{P}^{\mu} \left( \sup_{s \leq t} |Y_s| \geq R\right) \leq \mbb{P}^{\mu}\big(v(Y_{t \wedge \sigma_R}) \geq v(R) \big) \leq \frac{1}{v(R)} e^{Ct} \mbb{E}^{\mu}v(Y_0).
	\end{equation*}
	In particular we can choose for any compact set $K \subseteq \mbb{R}$ and any $\eps>0$ some $R>0$ such that \begin{equation*}
		\mbb{P}^{x} \left( \sup_{s \leq t} |Y_s|  \geq R \right) \leq \frac{\eps}{2} \fa x \in K.
	\end{equation*}
	Now if $\mu$ is an initial distribution such that $\mu(K^c) \leq \eps/2$, then, by \eqref{p-eq8}, \begin{align*}
		\mbb{P}^{\mu} \left( \sup_{s \leq t} |Y_s| \geq R\right)
		&= \int_K \mbb{P}^x \left( \sup_{s \leq t} |Y_s| \geq R \right) \, \mu(dx) + \int_{K^c} \mbb{P}^x \left( \sup_{s \leq t} |Y_s| \geq R \right) \, \mu(dx) \\
		&\leq \frac{\eps}{2} + \frac{\eps}{2}. \qedhere
	\end{align*}
\end{proof}

For the proof of Theorem~\ref{1.1} we will use the following result which follows e.\,g.\ from \cite[Theorem 4.1.16, Proof of Corollary 4.6.4]{jac3}.

\begin{lem} \label{p-9}
	Let $A$ be a pseudo-differential operator with negative definite symbol $q$ such that $A:C_c^{\infty}(\mbb{R}^d) \to C_b(\mbb{R}^d)$. If the $(A,C_c^{\infty}(\mbb{R}^d))$-martingale problem is well-posed and the unique solution $(X_t)_{t \geq 0}$ satisfies the compact containment condition \begin{equation*}
		\sup_{x \in K} \mbb{P}^x \left( \sup_{s \leq t} |X_s| \geq r \right) \xrightarrow[]{r \to \infty} 0
	\end{equation*}
	for any compact set $K \subseteq \mbb{R}^d$, then $x \mapsto \mbb{E}^x f(X_t)$ is continuous for all $f \in C_b(\mbb{R}^d)$.
\end{lem}

Now we are ready to prove Theorem~\ref{1.1}.

\begin{proof}[Proof of Theorem~\ref{1.1}]
	The well-posedness implies that the solution $(X_t)_{t \geq 0}$ is a Markov process, see e.\,g.\ \cite[Theorem 4.4.2]{ethier}, and by Corollary~\ref{p-5} the (unique) solution is conservative. In order to prove that $(X_t)_{t \geq 0}$ is a Feller process, we have to show that the semigroup $T_t f(x) := \mbb{E}^x f(X_t)$, $f \in C_{\infty}(\mbb{R}^d)$, has the following properties, cf.\ \cite[Lemma 1.4]{ltp}: \begin{enumerate}
		\item\label{i} continuity at $t=0$: $T_t f(x) \to f(x)$ as $t \to 0$ for any $x \in \mbb{R}^d$ and $f \in C_{\infty}(\mbb{R}^d)$.
		\item\label{ii} Feller property: \ $T_t (C_{\infty}(\mbb{R}^d)) \subseteq C_{\infty}(\mbb{R}^d)$ for all $t \geq 0$. 
	\end{enumerate}
	The first property is a direct consequence of the right-continuity of the sample paths and the dominated convergence theorem. Since we know that the martingale problem is well posed, it suffices to construct a solution to the martingale problem satisfying \ref{ii}. Write $\nu(x,dy) = \nu_s(x,dy)+\nu_l(x,dy)$ where \begin{align*}
		\nu_s(x,B) := \int_{|y|  < 1 \vee |x|/2} \I_B(y) \, \nu(x,dy) \qquad \qquad \nu_l(x,B) := \int_{|y| \geq 1 \vee |x|/2} \I_B(y) \, \nu(x,dy)
	\end{align*}
	are the small jumps and large jumps, respectively, and denote by $p$ the symbol with characteristics $(b,Q,\nu_s)$. By Corollary~\ref{p-5} there exists for any initial distribution $\mu$ a conservative solution to the $(p(x,D),C_c^{\infty}(\mbb{R}^d))$-martingale problem, and the solution satisfies \ref{p-7}\ref{p-7-ii} and \ref{p-7}\ref{p-7-iii}. Using the same reasoning as in \cite[proof of Proposition 4.10.2]{ethier} it is possible to show that we can use interlacing to construct a solution to the $(A,C_c^{\infty}(\mbb{R}^d))$-martingale problem with initial distribution $\mu=\delta_x$: \begin{equation*}
		X_t := \sum_{k \geq 0} Y_{t-\tau_k}^{(k)} \I_{[\tau_k,\tau_{k+1})}(t)
	\end{equation*}
	where \begin{itemize}
		\item $\tau_k := \inf\{t \geq 0; N_t = k\} =\sum_{j=1}^k \sigma_j$ are the jump times of a Poisson process $(N_t)_{t \geq 0}$ with intensity $\lambda := \sup_{z \in \mbb{R}^d} \nu_l(z,\mbb{R}^d \backslash \{0\})$, i.\,e.\ $\sigma_j \sim \Exp(\lambda)$ are independent and identically distributed. Note that $\lambda < \infty$ by Lemma~\ref{p-3}.
		\item $(Y^{(k,\mu_k)}_t)_{t \geq 0} := (Y^{(k)}_t)_{t \geq 0}$ is a solution to the $(p(x,D),C_c^{\infty}(\mbb{R}^d))$-martingale problem with initial distribution \begin{equation}
			\mu_k(B) := \frac{1}{\lambda} \mbb{E}^x \bigg( \int \I_B(z+y) \, \nu_l(z,dy) + (\lambda-\nu_l(z,\mbb{R}^d \backslash \{0\})) \delta_z(B) \bigg|_{z=Y_{\sigma_{k-1}-}^{(k-1)}} \bigg) \label{p-eq14}
		\end{equation} 
		for $k \geq 1$ and $\mu_0(dy) := \delta_x(dy)$. Moreover, $Y^{(k)}$ and $(\sigma_j)_{j \geq k+1}$ are independent for all $k \geq 0$.
		\item $\mbb{P}^x$ is a probability measure which depends on the initial distribution $\mu=\delta_x$ of $(X_t)_{t \geq 0}$.
	\end{itemize}
		Note that if we define a linear operator $P$ by \begin{equation}
			Pf(z) := \int f(z+y) \, \nu_l(z,dy) + (\lambda-\nu_l(z,\mbb{R}^d \backslash \{0\})) f(z), \qquad f \in C_{\infty}(\mbb{R}^d), \; z \in \mbb{R}^d \label{p-eq16}
	\end{equation}
	then \eqref{p-eq10} implies that \begin{equation}
		\mbb{E}^x f(Y_0^{(k)}) = \frac{1}{\lambda} \mbb{E}^x (Pf(Y_{\sigma_{k-1}-}^{(k-1)})) \fa f \in \mc{B}_b(\mbb{R}^d), k \geq 1. \label{p-eq14'}
	\end{equation}
	Before we proceed with the proof, let us give a remark on the construction of $(X_t)_{t \geq 0}$. The intensity of the Poisson process $(N_t)_{t \geq 0}$, which announces the ``large jumps'', is $\lambda = \sup_z \lambda(z)$ where $\lambda(z) := \nu_l(z,\mbb{R}^d \backslash \{0\})$ is the ``state-space dependent intensity'' of the large jumps. Roughly speaking the second term on the right-hand side of \eqref{p-eq14} is needed to thin out the large jumps; with probability $\lambda^{-1} \mbb{E}^x((\lambda-\lambda(Y_{\sigma_{k-1}-}^{(k-1)}))$ there is no large jump at time $\sigma_{k-1}$, and therefore the effective jump intensity at time $t=\sigma_{k-1}$ is $\lambda(Y_{\sigma_{k-1}-}^{(k-1)})$. \par \medskip
	We will prove that $(X_t)_{t \geq 0}$ has the Feller property. To this end, we first show that for any $t \geq 0$, $\eps>0$, $k \geq 1$ and any compact set $K \subseteq \mbb{R}^d$ there exists $R>0$ such that \begin{equation}
		\mbb{P}^x \left( \sup_{s \leq t} |Y_s^{(j,\mu_j)}| \geq R \right) \leq \eps \fa x \in K, j=0,\ldots,k; \label{p-eq145}
	\end{equation}
	we prove \eqref{p-eq145} by induction. Note that $\mu_j=\mu_j(x)$ depends on the initial distribution of $(X_t)_{t \geq 0}$. \begin{itemize}
		\item $k=0$: This is a direct consequence of Lemma~\ref{p-7}\ref{p-7-ii} since $\mu_0(dy) = \delta_x(dy)$.
		\item $k \to k+1$: Because of Lemma~\ref{p-7}\ref{p-7-ii} and the induction hypothesis, it suffices to show that there exists a compact set $C \subseteq \mbb{R}^d$ such that $\mbb{P}^x(Y_0^{(k+1,\mu_{k+1})} \notin C) \leq \eps/2$ for all $x \in K$. Choose $m \geq 0$ sufficiently large such that $\mbb{P}^x(\sigma_k \geq m) \leq \eps' := \eps/8$, and choose $R>0$ such that \eqref{p-eq145} holds with $\eps := \eps'$, $t:=m$. Then, by \eqref{p-eq14'} and our choice of $R$, 
		\begin{align*}
			\mbb{P}^x(|Y_0^{(k+1)}| \geq r) 
			&= \frac{1}{\lambda} \mbb{E}^x \left( (P\I_{\overline{B(0,r)}^c})(Y_{\sigma_k-}^{(k)}) \right) \\
			&\leq \epsilon' +  \frac{1}{\lambda} \mbb{E}^x \left( \I_{\{\sup_{s \leq m} |Y_s^{(k)}| \leq R\}} (P\I_{\overline{B(0,r)}^c})(Y_{\sigma_k-}^{(k)}) \right)
		\end{align*}
		which implies for $r>R$, $x \in K$
		 \begin{align*}
			&\quad \mbb{P}^x(|Y_0^{(k+1)}| \geq r) \\
			&\leq \eps' + \frac{1}{\lambda} \mbb{E}^x \left( \I_{\{\sup_{s \leq m} |Y_s^{(k)}| \leq R\}} \left[ \int \I_{B(0,r)^c}(Y_{\sigma_k-}^{(k)}+y)\, \nu_l(Y_{\sigma_k-}^{(k)},dy) + 2 \lambda \I_{B(0,r)^c}(Y_{\sigma_k-}^{(k)}) \right] \right) \\
			&\leq 3\eps' + \frac{1}{\lambda} \mbb{E}^x \left(\I_{\{\sup_{s \leq m} |Y_s^{(k)}| \leq R\}}  \int_0^m \!\! \int_{|y| \geq r-R} \, \nu(Y_{t-}^{(k)},dy) \, \mbb{P}_{\sigma_k}^x(dt) \right) \\
			&\leq 3\eps' + \frac{1}{\lambda} \sup_{|z| \leq R} \nu(z,B(0,r-R)^c).
		\end{align*}
		The second term on the right-hand side converges to $0$ as $r \to \infty$, cf.\  \cite[Theorem 4.4]{rs98} or \cite[Theorem A.1]{change}, and therefore we can choose $r>0$ sufficiently large such that $\mbb{P}^x(|Y_0^{(k+1)}| \geq r) \leq 4\eps' = \eps/2$ for all $x \in K$.
	\end{itemize}
	For fixed $\eps>0$ choose $k \geq 1$ such that $\mbb{P}^x(N_t \geq k+1) \leq \eps$. By definition of $(X_t)_{t \geq 0}$ and \eqref{p-eq145}, we get
	\begin{align*}
		\sup_{x \in K} \mbb{P}^x \left( \sup_{s \leq t} |X_s| \geq R\right) &\leq \sup_{x \in K} \mbb{P}^x \left( \bigcup_{j=0}^k \left\{ \sup_{s \leq t} \left|Y_s^{(j,\mu_j)}\right| \geq R \right\} \right) + \eps \leq (k+1) \eps.
	\end{align*}
	Thus, by Lemma~\ref{p-9}, $x \mapsto T_t f(x) = \mbb{E}^x f(X_t)$ is continuous for any $f \in C_{\infty}(\mbb{R}^d)$. It remains to show that $T_t f$ vanishes at infinity; to this end we will show that for any $r>0$, $\eps>0$ there exists a constant $M>0$ such that \begin{equation}
		\mbb{P}^x \left( \inf_{s \leq t} |X_s| <r  \right) \leq \eps \fa |x| \geq M.  \label{p-eq15}
	\end{equation}
	It follows from Lemma~\ref{p-3} and the very definition of $\lambda$ that $Pf$ defined in \eqref{p-eq16} is bounded and \begin{align*}
		|Pf(x)|
		&\leq \int_{|x+y| < r} |f(x+y)| \, \nu_l(x,dy)+ \int_{|x+y| \geq r} |f(x+y)| \, \nu_l(x,dy) + 2 \lambda |f(x)| \\
		&\leq \|f\|_{\infty} \nu(x,B(-x,r)) + \lambda \sup_{|z| \geq r} |f(z)| + 2 \lambda |f(x)| \\
		&\xrightarrow[]{|x| \to \infty} \lambda \sup_{|z| \geq r} |f(z)| \xrightarrow[]{r \to \infty} 0,
	\end{align*}
	i.\,e.\ $Pf$ vanishes at infinity for any $f \in C_{\infty}(\mbb{R}^d)$. We claim that for any $k \geq 0$, $\eps>0$, $t \geq 0$ and $r>0$ there exists a constant $M>0$ such that \begin{equation}
		\mbb{P}^x \left( \inf_{s \leq t} |Y_s^{(j,\mu_j)}|<r \right) \leq \eps \fa j=0,\ldots,k, |x| \geq M. \label{p-eq17}
	\end{equation}
	We prove \eqref{p-eq17} by induction. \begin{itemize}
		\item $k=0$: This follows from Lemma~\ref{p-7}\ref{p-7-ii} since $\mu_0(dy) = \delta_x(dy)$.
		\item $k \to k+1$: For fixed $r>0$ choose $\delta>0$ and $\varrho>0$ as in \ref{p-7}\ref{p-7-ii}. By \ref{p-7}\ref{p-7-ii} it suffices to show that there exists $M>0$ such that \begin{equation}
			\mu_{k+1}(B(0,\varrho)) \leq \delta \fa |x| \geq M. \label{p-star2} \tag{$\star$}
		\end{equation}
		(Note that $\mu_{k+1}=\mu_{k+1}(x)$ depends on the initial distribution of $(X_t)_{t \geq 0}$.) Pick a cut-off function $\chi \in C_c^{\infty}(\mbb{R}^d)$ such that $\I_{B(0,\varrho)} \leq \chi \leq \I_{B(0,\varrho+1)}$, then by \eqref{p-eq14}, \begin{align*}
			\mu_{k+1}(B(0,\varrho))
			\leq \mbb{E}^x \chi(Y_0^{(k+1,\mu_{k+1})})
			= \frac{1}{\lambda} \mbb{E}^x \big((P\chi)(Y_{\sigma_{k}-}^{(k,\mu_k)}) \big).
		\end{align*}
		If $\|P\chi\|_{\infty} =0$ this proves \eqref{p-star2}. If $\|P\chi\|_{\infty}>0$, then we can choose $m \geq 1$ such that $\mbb{P}^x(\sigma_1 \geq m) \leq \delta/(2\|P\chi\|_{\infty})$. Since $P\chi$ vanishes at infinity, we have $\sup_{|z| \geq R} |P\chi(z)| \leq \lambda \delta/4$ for $R>0$ sufficiently large. By the induction hypothesis, there exists $M>0$ such that \eqref{p-eq17} holds with $\eps := \lambda \delta/4$, $r:=R$ and $t := m$. Then \begin{align*}
			|\mbb{E}^x(P \chi)(Y_{s-}^{(k,\mu_k)})|
			&\leq \mbb{P}^x \left( |Y_{s-}^{(k,\mu_k)}|< R \right) \|P\chi\|_{\infty} + \sup_{|z| \geq R} |P \chi(z)| \leq \frac{1}{2} \lambda \delta
		\end{align*}
		for all $s \leq m$ and $|x| \geq M$, and therefore \begin{align*}
			\mu_{k+1}(B(0,\varrho))
			&= \frac{1}{\lambda} \mbb{E}^x (P\chi)(Y_{\sigma_{k}-}^{(k,\mu_k)}) \\
			&\leq \frac{1}{\lambda} \mbb{E}^x \left( \int_{(0,\infty)} P\chi(Y_{s-}^{(k,\mu_k)}) \, \mbb{P}^x_{\sigma_k}(ds) \right) \\
			&\leq \frac{\delta}{2} + \|P \chi\|_{\infty} \int_{(m,\infty)} \, \mbb{P}^x_{\sigma_1}(ds) \leq \delta.
		\end{align*}
	\end{itemize}
	For fixed $\eps>0$ and $t \geq 0$ choose $k \geq 1$ such that $\mbb{P}^x(N_t \geq k+1) \leq \eps$. Choose $M>0$ as in \eqref{p-eq17}, then \begin{equation*}
		\mbb{P}^x(|X_t| <R) \leq \mbb{P}^x \left( \bigcup_{j=0}^k \left\{ \inf_{s \leq t} |Y_s^{(j)}| < R \right\} \right) + \eps \leq 2\eps \fa |x| \geq M.
	\end{equation*}
	Consequently, we have shown that $(X_t)_{t \geq 0}$ is a Feller process. Since $(X_t)_{t \geq 0}$ solves the $(A,C_c^{\infty}(\mbb{R}^d))$-martingale problem, we have \begin{equation*}
		\mbb{E}^x u(X_{t \wedge \tau_r^x})-u(x) = \mbb{E}^x \left( \int_{(0,t \wedge \tau_r^x)} Au(X_s) \, ds \right), \qquad u \in C_c^{\infty}(\mbb{R}^d),
	\end{equation*}
	where $\tau_r^x := \inf\{t \geq 0; |X_t-x| \geq r\}$ denotes the exit time from the ball $B(x,r)$. Using that $A(C_c^{\infty}(\mbb{R}^d)) \subseteq C_{\infty}(\mbb{R}^d)$, it is not difficult to see that the generator of $(X_t)_{t \geq 0}$ is, when restricted to $C_c^{\infty}(\mbb{R}^d)$, a pseudo-differential operator with symbol $q$, see e.\,g.\ \cite[Proof of Theorem 3.5, Step 2]{sde} for details. This means that $(X_t)_{t \geq 0}$ is a rich Feller process with symbol $q$.
\end{proof}

\begin{proof}[Proof of Corollary~\ref{1.2}]
	By Corollary~\ref{p-5} there exists for any initial distribution $\mu$ a solution to the $(A,C_c^{\infty}(\mbb{R}^d))$-martingale problem, and by assumption the martingale problem for the pseudo-differential operator $A_k$ with symbol $q_k$ is well-posed. Therefore \cite[Theorem 5.3]{hoh}, see also \cite[Theorem 4.6.2]{ethier}, shows that the $(A,C_c^{\infty}(\mbb{R}^d))$-martingale problem is well-posed. Now the assertion follows from Theorem~\ref{1.1}.
\end{proof}

\section{Applications} \label{app}

In this section we apply our results to L\'evy-driven stochastic differential equations (SDEs) and stable-like processes. Corollary~\ref{p-5} gives the following general existence result for weak solutions to L\'evy-driven SDEs. 

\begin{thm} \label{app-0}
	Let $(L_t)_{t \geq 0}$ be a $k$-dimensional L\'evy process with characteristic exponent $\psi$ and L\'evy triplet $(b,Q,\nu)$. Let $\ell: \mbb{R}^d \to \mbb{R}^d$, $\sigma:\mbb{R}^d \to \mbb{R}^{d \times k}$ be continuous functions which grow at most linearly. If \begin{equation}
		\nu(\{y \in \mbb{R}^k; |\sigma(x) \cdot y+ x| \leq r\}) \xrightarrow[]{|x| \to \infty} 0 \fa r>0, \label{app-eq0}
	\end{equation}
	then the SDE \begin{equation}
		dX_t = \ell(X_{t-}) \, dt  + \sigma(X_{t-}) \, dL_t, \qquad X_0 \sim \mu \label{app-eq1}
	\end{equation}
	has for any initial distribution $\mu$ a weak solution $(X_t)_{t \geq 0}$ which is conservative.
\end{thm}

Note that \eqref{app-eq0} is, in particular, satisfied if \begin{equation*}
	\lim_{|x| \to \infty} \sup_{|\xi| \leq |x|^{-1}} |\re \psi(\sigma(x)^T \xi)| = 0,
\end{equation*}
e.\,g.\ if $\sigma$ is at most of sublinear growth, cf.\ Lemma~\ref{map}\ref{map-ii}.

\begin{proof}
	Denote by $A$ the pseudo-differential operator with symbol $q(x,\xi) :=-i \ell(x) \cdot \xi + \psi(\sigma(x)^T \xi)$. Since $q$ is locally bounded and $x \mapsto q(x,\xi)$ is continuous for all $\xi \in \mbb{R}^d$ it follows from \eqref{app-eq1} that $A(C_c^{\infty}(\mbb{R}^d)) \subseteq C_{\infty}(\mbb{R}^d)$, cf.\ Lemma~\ref{map}. Because $\ell$, $\sigma$ are at most of linear growth,  $q$ satisfies the growth condition \eqref{lin-grow}. Applying Corollary~\ref{p-5} we find that there exists a conservative solution $(X_t)_{t \geq 0}$ to the $(A,C_c^{\infty}(\mbb{R}^d))$-martingale problem. By \cite{kurtz}, $(X_t)_{t \geq 0}$ is a weak solution to the SDE \eqref{app-eq1}.
\end{proof}

For $\alpha \in (0,1]$ we denote by \begin{align*}
	\mc{C}^{\alpha}_{\loc}(\mbb{R}^d,\mbb{R}^n) &:= \left\{f: \mbb{R}^d \to \mbb{R}^n; \forall x \in \mbb{R}^d: \sup_{|y-x| \leq 1} \frac{|f(y)-f(x)|}{|y-x|^{\alpha}}<\infty \right\} \\
	\mc{C}^{\alpha}(\mbb{R}^d,\mbb{R}^n) &:= \left\{f: \mbb{R}^d \to \mbb{R}^n; \sup_{x \neq y} \frac{|f(y)-f(x)|}{|y-x|^{\alpha}}<\infty \right\}
\end{align*}
the space of (locally) H\"{o}lder continuous functions with H\"{o}lder exponent $\alpha$.


\begin{thm} \label{app-1} 
	Let $(L_t)_{t \geq 0}$ be a $k$-dimensional L\'evy process with L\'evy triplet $(b,Q,\nu)$ and characteristic exponent $\psi$. Suppose that there exist $\alpha,\beta \in (0,1]$ such that the L\'evy-driven SDE \begin{equation*}
		dX_t = f(X_{t-}) \, dt+ g(X_{t-}) \, dL_t, \qquad X_0 \sim \mu 
	\end{equation*}
	has a unique weak solution for any initial distribution $\mu$ and any two bounded functions $f \in \mc{C}^{\alpha}(\mbb{R}^d,\mbb{R}^d)$ and $g \in \mc{C}^{\beta}(\mbb{R}^d,\mbb{R}^{d \times k})$ such that \begin{equation*}
		|g(x)^T \xi| \geq c |\xi|, \qquad \xi \in \mbb{R}^d, x \in \mbb{R}^d
	\end{equation*}
	for some constant $c>0$. Then the SDE \begin{equation*}
		dX_t = \ell(X_{t-}) \, dt+ \sigma(X_{t-}) \, dL_t, \qquad X_0 \sim \mu
	\end{equation*}
	has a unique weak solution for any $\ell \in C^{\alpha}_{\loc}(\mbb{R}^d,\mbb{R}^d)$, $\sigma \in C^{\beta}_{\loc}(\mbb{R}^d,\mbb{R}^{d \times k})$ which are at most of linear growth and satisfy 
	\begin{equation}
			\nu(\{y \in \mbb{R}^k; |\sigma(x) \cdot y + x| \leq r\}) \xrightarrow[]{|x| \to \infty} 0 \fa r>0 \label{app-eq5}
		\end{equation}
	and 
	\begin{equation}
		\forall n \in \mbb{N} \, \, \exists c_n>0 \, \, \forall |x| \leq n, \xi \in\mbb{R}^d: \, \, \,  |\sigma(x)^T \xi| \geq c_n |\xi|. \label{app-eq7}
	\end{equation}
	The unique weak solution is a conservative rich Feller process with symbol \begin{equation*}
		q(x,\xi) :=-i \ell(x) \cdot  \xi + \psi(\sigma(x)^T \xi), \qquad x,\xi \in \mbb{R}^d.
	\end{equation*}
\end{thm}

\begin{proof}
	Let $\ell \in \mc{C}^{\alpha}_{\loc}(\mbb{R}^d,\mbb{R}^d)$ and $\sigma \in \mc{C}^{\beta}_{\loc}(\mbb{R}^d,\mbb{R}^{d \times k})$ be two functions which grow at most linearly and satisfy \eqref{app-eq5}, \eqref{app-eq7}. Lemma~\ref{map} shows that the pseudo-differential operator $A$ with symbol $q$ satisfies $A(C_c^{\infty}(\mbb{R}^d)) \subseteq C_{\infty}(\mbb{R}^d)$. Moreover, since $\sigma$, $\ell$ are at most of linear growth, the growth condition \eqref{lin-grow} is clearly satisfied. Set \begin{equation*}
		\ell_k(x) := \begin{cases} \ell(x), & |x| <k  \\ \ell \left( k \frac{x}{|x|} \right), & |x| \geq k \end{cases} \quad \text{and} \quad \sigma_k(x) := \begin{cases} \sigma(x), &  |x|<k, \\ \sigma\left( k \frac{x}{|x|} \right), & |x| \geq k. \end{cases} 
	\end{equation*} 
	By assumption, the SDE \begin{equation*}
		dX_t = \ell_k(X_{t-}) \, dt + \sigma_k(X_{t-}) \, dL_t, \qquad X_0 \sim \mu,
	\end{equation*}
	has a unique weak solution for any initial distribution $\mu$ for all $k \geq 1$. By \cite{kurtz} (see also \cite[Lemma 3.3]{sde}) this implies that the $(A_k,C_c^{\infty}(\mbb{R}^d))$-martingale problem for the pseudo-differential operator with symbol $q_k(x,\xi) := -i \ell_k(x) \cdot  \xi + \psi(\sigma_k(x)^T \xi)$ is well-posed. Since $\sigma_k$ is bounded, we have \begin{equation*}
		\nu(\{y \in \mbb{R}^k; |\sigma_k(x) \cdot y +x| \leq r\}) \xrightarrow[]{|x| \to \infty} 0 \fa r>0, 
	\end{equation*}
	and therefore Lemma~\ref{map} shows that $A_k$ maps $C_c^{\infty}(\mbb{R}^d)$ into $C_{\infty}(\mbb{R}^d)$. Now the assertion follows from Corollary~\ref{1.2}.
\end{proof}

Applying Theorem~\ref{app-1} we obtain the following generalization of \cite[Corollary 4.7]{parametrix}, see also \cite[Theorem 5.23]{matters}. 

\begin{thm} \label{app-3}
	Let $(L_t)_{t \geq 0}$ be a one-dimensional L\'evy process such that its characteristic exponent $\psi$ satisfies the following conditions: \begin{enumerate}
		\item $\psi$ has a holomorphic extension $\Psi$ to \begin{equation*}
			U := \{z \in \mbb{C}; |\im z| < m\} \cup \{z \in \mbb{C} \backslash \{0\}; \arg z \in (-\vartheta,\vartheta) \cup (\pi-\vartheta,\pi+\vartheta)\}
		\end{equation*}
		for some $m \geq 0$ and $\vartheta \in (0,\pi/2)$.
		\begin{figure}[H]
							\begin{center}
								\includegraphics[scale=0.75]{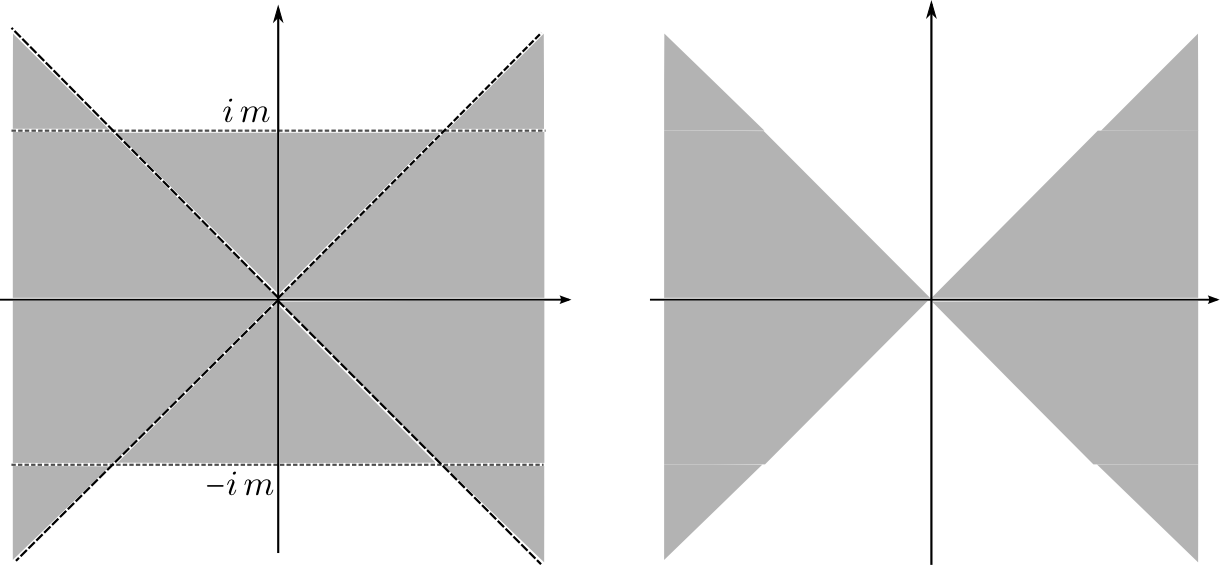}
							\end{center}
							\caption{The domain $U = U(m,\vartheta)$ for $m>0$ (left) and $m=0$ (right).}
							\label{fig:def_gebiet_exp}
		\end{figure} 
		\item There exist $\alpha \in (0,2]$, $\beta  \in (1,2)$ and constants $c_1,c_2>0$ such that \begin{equation*}
				\re \Psi(z) \geq c_1 |\re z|^{\beta} \fa z \in U, \; |z| \gg 1,
		\end{equation*}
		and \begin{equation*}
			|\Psi(z)| \leq c_2 (|z|^{\alpha} \I_{\{|z| \leq 1\}} + |z|^{\beta} \I_{\{|z|>1\}}), \qquad z \in U.
		\end{equation*}
		\item There exists a constant $c_3>0$ such that $|\Psi'(z)| \leq c_3 |z|^{\beta-1}$ for all $z \in U$, $|z| \gg 1$.
	\end{enumerate}
	Let $\ell: \mbb{R} \to \mbb{R}$ and $\sigma:\mbb{R} \to (0,\infty)$ be two locally H\"older continuous functions which grow at most linearly. If \begin{equation*}
		\nu(\{x; |\sigma(x)y+x| \leq r\}) \xrightarrow[]{|x| \to \infty} 0 \fa r>0,
	\end{equation*}
	 then the SDE \begin{equation*}
		dX_t = \ell(X_{t-}) \, dt + \sigma(X_{t-}) \, dL_t, \qquad X_0 \sim \mu,
	\end{equation*}
	has a unique weak solution for any initial distribution $\mu$. The unique solution is a conservative rich Feller process with symbol $q(x,\xi) := -i\ell(x) \xi + \psi(\sigma(x) \xi)$.
\end{thm}

\begin{proof}
	\cite[Corollary 4.7]{parametrix} shows that the assumptions of Theorem~\ref{app-1} are satisfied, and this proves the assertion.
\end{proof}

Theorem~\ref{app-3} applies, for instance, to L\'evy processes with the following characteristic exponents:  \begin{enumerate}
	\item (isotropic stable) $\psi(\xi) = |\xi|^{\alpha}$, $\xi \in \mbb{R}$, $\alpha \in (1,2]$,
	\item (relativistic stable) $\psi(\xi) = (|\xi|^2+\varrho^2)^{\alpha/2}-\varrho^{\alpha}$, $\xi \in \mbb{R}$, $\varrho>0$, $ \alpha \in (1,2)$,
	\item (Lamperti stable) $\psi(\xi) = (|\xi|^2+\varrho)_{\alpha}-(\varrho)_{\alpha}$, $\xi \in \mbb{R}$, $\varrho>0$, $\alpha \in (1/2,1)$, where $(r)_{\alpha} := \Gamma(r+\alpha)/\Gamma(r)$ denotes the Pochhammer symbol,
	\item (truncated L\'evy process) $\psi(\xi) = (|\xi|^2+\varrho^2)^{\alpha/2} \cos(\alpha \arctan(\varrho^{-1} |\xi|))-\varrho^{\alpha}$, $\xi \in \mbb{R}$, $\alpha \in (1,2)$, $\varrho>0$,
	\item (normal tempered stable) $\psi(\xi) = (\kappa^2+(\xi-ib)^2)^{\alpha/2}-(\kappa^2-b^2)^{\alpha/2}$, $\xi \in \mbb{R}$, $\alpha \in (1,2)$, $b>0$, $|\kappa|>|b|$.
\end{enumerate}
For further examples of L\'evy processes satisfying the assumptions of Theorem~\ref{app-3} we refer to \cite{parametrix,matters}. \par \medskip

We close this section with two further applications of Corollary~\ref{1.2}. The first is an existence result for Feller processes with symbols of the form $p(x,\xi)= \varphi(x) q(x,\xi)$. Recall that $p(x,D)$ denotes the pseudo-differential operator with symbol $p$.

\begin{thm} \label{app-5}
	Let $A$ be a pseudo-differential operator with symbol $q$ such that $q(\cdot,0)=0$, $A(C_c^{\infty}(\mbb{R}^d)) \subseteq C_{\infty}(\mbb{R}^d)$ and \begin{equation*}
		\lim_{|x| \to \infty} \sup_{|\xi| \leq |x|^{-1}} |q(x,\xi)| < \infty.
	\end{equation*}
	 Assume that for any continuous bounded function $\sigma: \mbb{R}^d \to (0,\infty)$ the $(\sigma(x) q(x,D),C_c^{\infty}(\mbb{R}^d))$-martingale problem for the pseudo-differential operator with symbol $\sigma(x) q(x,\xi)$ is well-posed. If $\varphi: \mbb{R}^d \to (0,\infty)$ is a continuous function such that \begin{equation}
		 \lim_{|x| \to \infty} \sup_{|\xi| \leq |x|^{-1}} \big( \varphi(x) |q(x,\xi)| \big)<\infty, \label{app-eq11}
		\end{equation}
		and \begin{equation}
			\varphi(x) \nu(x,B(-x,r)) \xrightarrow[]{|x| \to \infty} 0 \fa r>0, \label{app-eq13}
		\end{equation}
		then there exists a conservative rich Feller process $(X_t)_{t \geq 0}$ with symbol $p(x,\xi) := \varphi(x) q(x,\xi)$ and $(X_t)_{t \geq 0}$ is the unique solution to the $(p(x,D),C_c^{\infty}(\mbb{R}^d))$-martingale problem.
\end{thm}

Theorem~\ref{app-5} is more general than \cite[Theorem 4.6]{change}. \emph{Indeed:} If there exists a rich Feller process $(X_t)_{t \geq 0}$ with symbol $q$ and $C_c^{\infty}(\mbb{R}^d)$ is a core for the infinitesimal generator of $(X_t)_{t \geq 0}$, then, by \cite[Theorem 4.2]{ltp}, there exists for any continuous bounded function $\sigma>0$ a rich Feller process with symbol $\sigma(x) q(x,\xi)$ and core $C_c^{\infty}(\mbb{R}^d)$, and therefore the $(\sigma(x) q(x,D),C_c^{\infty}(\mbb{R}^d))$-martingale problem is well-posed, cf.\ \cite[Theorem 4.10.3]{kol}.

\begin{proof}[Proof of Theorem~\ref{app-5}]
	For given $\varphi$ define \begin{equation*}
		\varphi_k(x) := \varphi(x) \I_{B(0,k)}(x) + \varphi \left( k \frac{x}{|x|} \right) \I_{B(0,k)^c}(x).
	\end{equation*}
	By assumption, the $(\varphi_k(x) q(x,D),C_c^{\infty}(\mbb{R}^d))$-martingale problem is well-posed. Moreover, it follows from the boundedness of $\varphi_k$ and the fact that $q(x,D)(C_c^{\infty}(\mbb{R}^d)) \subseteq C_{\infty}(\mbb{R}^d)$ that $\varphi_k(x) q(x,D)$ maps $C_c^{\infty}(\mbb{R}^d)$ into $C_{\infty}(\mbb{R}^d)$. On the other hand, \eqref{app-eq13} gives $p(x,D)(C_c^{\infty}(\mbb{R}^d)) \subseteq C_{\infty}(\mbb{R}^d)$, cf.\ Lemma~\ref{map}. Applying Corollary~\ref{1.2} proves the assertion.
\end{proof}

\begin{bsp} \label{app-7}
	Let $\varphi: \mbb{R}^d \to (0,\infty)$ be a continuous fuction and $\alpha: \mbb{R}^d \to (0,2]$ a locally H\"{o}lder continuous function. If there exists a constant $c>0$ such that $\varphi(x) \leq c(1+|x|^{\alpha(x)})$ for all $x \in \mbb{R}^d$, then there exists a conservative rich Feller process $(X_t)_{t \geq 0}$ with symbol \begin{equation*}
		p(x,\xi) := \varphi(x) |\xi|^{\alpha(x)}, \qquad x,\xi \in \mbb{R}^d,
	\end{equation*}
	and $(X_t)_{t \geq 0}$ is the unique solution to the $(p(x,D),C_c^{\infty}(\mbb{R}^d))$-martingale problem. \par \medskip
	
	\emph{Indeed:} If we set \begin{equation*}
		\alpha_j(x) := \alpha(x) \I_{B(0,j)}(x) + \alpha \left( j \frac{x}{|x|} \right) \I_{B(0,j)^c}(x),
	\end{equation*}
	then \cite[Theorem 5.2]{diss} shows that there exists a rich Feller process with symbol $q_j(x,\xi) := |\xi|^{\alpha_j(x)}(x)$, and that $C_c^{\infty}(\mbb{R}^d)$ is a core for the generator. By \cite[Theorem 4.2]{ltp}, there exists for any continuous bounded function $\sigma>0$ a rich Feller process with symbol $\sigma(x) q_j(x,\xi)$ and core $C_c^{\infty}(\mbb{R}^d)$. This implies that the $(\sigma(x) q_j(x,D),C_c^{\infty}(\mbb{R}^d))$-martingale problem is well posed, see e.\,g.\ \cite[Theorem 4.10.3]{kol} or \cite[Theorem 1.37]{diss}. Applying Theorem~\ref{app-5} we find that there exists a conservative rich Feller process with symbol $p_j(x,\xi) := \varphi(x) q_j(x,\xi)$, and that the $(p_j(x,D),C_c^{\infty}(\mbb{R}^d))$-martingale problem is well-posed. Now the assertion follows from Corollary~\ref{1.2}.
\end{bsp}

Example~\ref{app-7} shows that Corollary~\ref{1.2} is useful to establish the existence of stable-like processes with unbounded coefficients. For relativistic stable-like processes we obtain the following general existence result.

\begin{thm} \label{app-9}
	Let $\alpha: \mbb{R}^d \to (0,2]$, $m: \mbb{R}^d \to (0,\infty)$ and $\kappa: \mbb{R}^d \to (0,\infty)$ be locally H\"older continuous functions. If \begin{equation}
		\sup_{|x| \geq 1} \frac{\kappa(x)}{|x|^2 m(x)^{2-\alpha(x)}} < \infty \label{app-eq17}
	\end{equation}
	and \begin{equation}
		\kappa(x) m(x) e^{-|x| m(x)/4} \xrightarrow[]{|x| \to \infty} 0, \label{app-eq19}
	\end{equation}
	then there exists a conservative rich Feller process $(X_t)_{t \geq 0}$ with symbol \begin{equation*}
		q(x,\xi) := \kappa(x) \left[ (|\xi|^2 + m(x)^2)^{\alpha(x)/2} -m(x)^{\alpha(x)} \right], \qquad x,\xi \in \mbb{R}^d,
	\end{equation*}
	and $(X_t)_{t \geq 0}$ is the unique solution to the $(q(x,D),C_c^{\infty}(\mbb{R}^d))$-martingale problem.
\end{thm} 

Note that $\kappa$ and $m$ do not need to be of linear growth; for instance if $\inf_x \alpha(x)>0$, then we can choose $m(x) := e^{|x|}$ and $\kappa(x) := (1+|x|^k)$ for $k \geq 1$.

\begin{proof}[Proof of Theorem~\ref{app-9}]
	For a function $f: \mbb{R}^d \to \mbb{R}$ set \begin{equation*}
		f_i(x) := f(x) \I_{B(0,i)}(x) + f \left(i \frac{x}{|x|} \right) \I_{B(0,i)^c}(x)
	\end{equation*}
	and define \begin{equation*}
		q_i(x,\xi) := \kappa_i(x) \left[ (|\xi|^2 + m_i(x)^2)^{\alpha_i(x)/2} -m_i(x)^{\alpha_i(x)} \right].
	\end{equation*}
	Since $\kappa_i$, $\alpha_i$ and $m_i$ are bounded H\"{o}lder continuous functions which are bounded away from $0$, it follows from \cite{matters}, see also \cite{diss}, that the $(q_k(x,D),C_c^{\infty}(\mbb{R}^d))$-martingale problem is well-posed. Consequently, the assertion follows from Corollary~\ref{1.2} if we can show that $q$ satisfies \eqref{lin-grow} and that the pseudo-differential operators $q(x,D)$ and $q_i(x,D)$, $i \geq 1$, map $C_c^{\infty}(\mbb{R}^d)$ into $C_{\infty}(\mbb{R}^d)$. An application of Taylor's formula yields \begin{align*}
		\sup_{|\xi| \leq |x|^{-1}} |q(x,\xi)|
		&\leq \kappa(x) \big[ (|x|^{-2} +m(x)^2)^{\alpha(x)/2}- (m(x)^2)^{\alpha(x)/2} \big] \\
		&\leq \kappa(x) \frac{1}{|x|^2} \frac{\alpha(x)}{2} m(x)^{\alpha(x)-2}, 
	\end{align*}
	and by \eqref{app-eq17} this implies \eqref{lin-grow}. It remains to prove the mapping properties of $q(x,D)$ and $q_i(x,D)$. Since $x \mapsto q_i(x,\xi)$ is continuous and \begin{equation*}
		\sup_{|\xi| \leq |x|^{-1}} |q(x,\xi)|
		\leq \|\kappa_i\|_{\infty} \left( \inf_{|x| \leq i} m(x) \right)^{-2} \frac{1}{|x|^2} \xrightarrow[]{|x| \to \infty} 0
	\end{equation*}
	it follows from Lemma~\ref{map} that $q_i(x,D)(C_c^{\infty}(\mbb{R}^d)) \subseteq C_{\infty}(\mbb{R}^d)$. To prove $q(x,D)(C_c^{\infty}(\mbb{R}^d)) \subseteq C_{\infty}(\mbb{R}^d)$ we note that $x \mapsto q(x,\xi)$ is continuous, and therefore it suffices to show by Lemma~\ref{map} that \begin{equation*}
		\lim_{|x| \to \infty} \nu(x,B(-x,r)) \xrightarrow[]{|x| \to \infty} 0, \qquad r>0,
	\end{equation*}
	where $\nu(x,dy)$ is for each fixed $x \in \mbb{R}^d$ the L\'evy measure of a relativistic stable L\'evy process with parameters $(\kappa(x),m(x),\alpha(x))$. It is known that $\nu(x,dy) \leq c \kappa(x) e^{-|y| m(x)/2} \, dy$ on $B(0,1)^c$, and therefore \begin{align*}
		\nu(x,B(-x,r))
		\leq c \kappa(x) \int_{B(-x,r)} e^{-|y| m(x)/2} \, dy
		= c \kappa(x) \left( e^{-|x-r| m(x)/2} -  e^{-|x+r| m(x)/2} \right).
	\end{align*}
	For $|x| \gg 1$ and fixed $r>0$ we obtain from Taylor's formula \begin{equation*}
		\nu(x,B(-x,r))
		\leq c \kappa(x) m(x) e^{-|x| m(x)/4} 
		\xrightarrow[\eqref{app-eq19}]{|x| \to \infty} 0. \qedhere
	\end{equation*}
\end{proof}

\begin{ack}
	I would like to thank Ren\'e Schilling for helpful comments and suggestions.
\end{ack}


\begin{thebibliography}{99}\frenchspacing	
	\bibitem{ltp}
	    B\"{o}ttcher, B., Schilling, R.\,L., Wang, J.: \emph{L\'evy-Type Processes: Construction, Approximation and Sample Path Properties}. Springer Lecture Notes in Mathematics vol.\ \textbf{2099}, (vol.~III of the ``L\'evy Matters'' subseries). Springer, 2014.
	\bibitem{ethier}
		Ethier, S.\,N., Kurtz, T.\,G.: \emph{Markov processes - characterization and convergence}. Wiley, 1986.
	\bibitem{hoh}
			Hoh, W.: \emph{Pseudo-Differential Operators Generating Markov Processes}. Habilitationsschrift. Universit\"{a}t Bielefeld, Bielefeld 1998.
	\bibitem{jac3}
		Jacob, N.: \emph{Pseudo Differential Operators and Markov Processes III}. Imperial College Press/World Scientific, London 2005.
	\bibitem{kol}
		Kolokoltsov, V.: \emph{Markov Processes, Semigroups and Generators}. De Gruyter, 2011.
	\bibitem{kurtz}
		Kurtz, T.\,G.: Equivalence of stochastic equations and martingale problems. In: Crisan, D. (ed.), \emph{Stochastic Analysis 2010}, Springer, 2011, pp.~113--130.
		\bibitem{sde}
			K{\"u}hn, F.: Solutions of L\'evy-driven SDEs with unbounded coefficients as Feller processes. Preprint arXiv 1610.02286.
	\bibitem{diss}
			K\"uhn, F.: Probability and Heat Kernel Estimates for L\'evy(-Type) Processes. PhD Thesis, Technische Universit{\"a}t Dresden 2016. http://nbn-resolving.de/urn:nbn:de:bsz:14-qucosa-214839
		\bibitem{parametrix}
			K\"uhn, F.: Transition probabilities of L\'evy-type processes: Parametrix construction. Preprint arXiv 1702.00778.
	\bibitem{change}
		K\"uhn, F.: Random time changes of Feller processes. Preprint arXiv 1705.02830.
		\bibitem{matters}
			K{\"u}hn, F.: \emph{L\'evy-Type Processes: Moments, Construction and Heat Kernel Estimates}. Springer Lecture Notes in Mathematics vol.\ \textbf{2187} (vol.~VI of the ``L\'evy Matters'' subseries). Springer, to appear.
	\bibitem{sato}
		Sato, K.-I.: \emph{L\'evy Processes and Infinitely Divisible Distributions}. Cambridge University Press, Cambridge 2005.
	\bibitem{rs98}
		Schilling, R.\,L.: Conservativeness and Extensions of Feller Semigroups. \emph{Positivity} \textbf{2} (1998), 239--256.
	\bibitem{rs-grow}
		Schilling, R.\,L.: Growth and H\"{o}lder conditions for the sample paths of Feller processes. \emph{Probab.\ Theory Relat.\ Fields} \textbf{112} (1998), 565--611.
	\bibitem{situ}
		Situ, R.: \emph{Theory of stochastic differential equations with jumps and applications}. Springer, 2005.
	\bibitem{cast1}
		van Casteren, J.\,A.: On martingales and Feller semigroups. \emph{Results in Mathematics} \textbf{21} (1992), 274--288.
\end{thebibliography}
\end{document}